\documentclass[12pt]{amsart}
\usepackage{extsizes}
\usepackage{amssymb, amsmath}
\usepackage{relsize}
\usepackage{fourier}

\usepackage{geometry}
\theoremstyle{plain}
\newtheorem{theorem}{Theorem}
\newtheorem*{theorem*}{Theorem}

\newtheorem{lemma}[theorem]{Lemma}

\theoremstyle{remark}

\newcommand{\B}{\mathcal{B}(H)}
\newcommand{\F}{\mathcal{F}(H)}
\newcommand{\K}{\mathcal{C}(H)}

\newcommand\tr{\operatorname{Tr}}
\newcommand{\C}{\mathbb{C}}

\newcommand{\N}{\mathbb{N}}

\begin{document}

\title[]{On 2-local *-automorphisms and 2-local isometries of $\B$}

\author{Lajos Moln\'ar}
\address{University of Szeged, Interdisciplinary Excellence Centre, Bolyai Institute,
H-6720 Szeged, Aradi v\'ertan\'uk tere 1.,
Hungary, and
Budapest University of Technology and Economics,  Institute of Mathematics,
H-1521 Budapest, Hungary}
\email{molnarl@math.u-szeged.hu, molnarl@math.bme.hu}
\urladdr{http://www.math.u-szeged.hu/\~{}molnarl}

\thanks{This paper was written while the author was a visiting researcher at the Alfr\'ed R\'enyi Institute of Mathematics (Hungarian Academy of Sciences). His research was supported by
the Ministry of Human Capacities, Hungary grant 20391-3/2018/FEKUSTRAT and by the National Research, Development and Innovation Office NKFIH, Grant No. K115383. The author expresses his thanks to L\'aszl\'o Szilas for his useful comments and technical help during this research.}

\begin{abstract}
It is an important result of \v Semrl which states that every 2-local automorphism of the full operator algebra over a separable Hilbert space is necessarily an automorphism. In this paper we strengthen that result quite substantially for *-automorphisms. Indeed, we show that one can compress the defining two equations of 2-local *-automorphisms into one single equation, hence weakening the requirement significantly, but still keeping essentially the conclusion that such maps are necessarily *-automorphisms.
\end{abstract}

\maketitle

\section{Introduction and Statements of the Results}

The concept of 2-local automorphisms of algebras was introduced by \v Semrl in the paper \cite{Sem} as follows. For an algebra $\mathcal A$ the map $\theta: \mathcal A\to \mathcal A$ (which is not assumed to be linear) is called a 2-local automorphism of $\mathcal A$ if for every $A,B\in \mathcal A$ there is an algebra automorphism $\theta_{A,B}$ of $\mathcal A$ (depending on $A,B$) such that
\begin{equation}\label{E:main}
\theta(A)=\theta_{A,B}(A) \quad \text{and} \quad\theta(B)=\theta_{A,B}(B).
\end{equation}
\v Semrl's motivation to introduce this concept originated from Kowalski and Slodkowski's version of the famous Gleason-Kahane-\.Zelazko theorem, see \cite{KowSlo}. Theorem 1 (also see Remark) in \cite{Sem} tells us the quite surprising observation that if $H$ is a separable Hilbert space, then every 2-local automorphism of algebra $\B$ is in fact an algebra automorphism. Here and in what follows, $\B$ denotes the $C^*$-algebra of all bounded linear operators on $H$. This remarkable result attracted serious attention and motivated a number of further investigations. We refer only to some of the related papers \cite{AlHFle, AyuKud1, AyuKud2, Bur, CW, Gyor,  HatMiuetal, Jim, VarVal, Kim1, Kim2, LiuWon, M3, M1, M2, XL} and Chapter 3 in the book \cite{M} which treats these kinds of problems.

The aim of the present paper is to show that, in a certain context, even more than what was obtained in \cite{Sem} can be proven. Namely, for algebra *-automorphisms of $\B$, the two equations appearing in \eqref{E:main} can be compressed into one single equation and we still obtain essentially the same conclusion as in \cite{Sem}.
More precisely, we prove the next theorem. 

Throughout this paper, $H$ stands for a separable complex Hilbert space.

\begin{theorem}\label{T:add}
Suppose that $\phi:\mathcal{B}(H)\to\mathcal{B}(H)$ is a map (linearity is not assumed) with the following property: for any $A,B \in \mathcal{B}(H)$ there exists a *-automorphism $\phi_{A,B}$ of the $C^*$-algebra $\mathcal{B}(H)$ such that
$$
\phi(A)+\phi(B)=\phi_{A,B}(A+B).
$$
If $\dim H\geq 3$, then $\phi$ is necessarily a *-automorphism of $\mathcal{B}(H)$. If $\dim H=2$, then $\phi$ is either a *-automorphism or a *-antiautomorphism of $\B$.
\end{theorem}

The previous result shows that, assuming $\dim H \geq 3$, we can sum up the equalities in \eqref{E:main} and still have the conclusion that such a map is necessarily a *-automorphism. As for  the operation of multiplication, we have a similar statement which reads as follows.

\begin{theorem}\label{T:szor}
Suppose that $\phi:\mathcal{B}(H)\to\mathcal{B}(H)$ is a map (linearity is not assumed) with the following property: for any $A,B \in \mathcal{B}(H)$ there exists a *-automorphism $\phi_{A,B}$ of the $C^*$-algebra $\mathcal{B}(H)$ such that
$$
\phi(A)\phi(B)=\phi_{A,B}(AB).
$$
Then $\phi$ is either a *-automorphism or the negative of a *-automorphism of $\mathcal{B}(H)$.
\end{theorem}

\section{Proofs}
In this section we present the proofs of our statements. We begin with Theorem \ref{T:add}. In fact, that statement will be deduced from the following somewhat stronger result concerning 2-local maps of $\B$ corresponding to its full group of all (not necessarily linear) isometries. (The fact that the next result is formally really stronger will be discussed below.)
The result concerns the isometries of $\B$ which correspond to the metric induced by the operator norm $\Vert.\Vert$. 

\begin{theorem}\label{T:izom}
Let $\phi:\B\to \B$ be a map (no linearity is assumed) with the property that for any $A,B\in \B$ there exists a surjective isometry (surjective distance preserving map) $\phi_{A,B}$ of $\B$ such that $\phi(A)=\phi_{A,B}(A),\phi(B)=\phi_{A,B}(B)$. Then $\phi$ is necessarily a surjective isometry of $\B$.
\end{theorem}

We remark that similar results concerning the group of all linear surjective isometries of operator algebras and function algebras were presented, among others, in the papers \cite{AlHFle, Gyor, HatMiuetal, Jim, VarVal, M1}. 

Turning to the statement in Theorem \ref{T:izom},
one can trivially see that, by the assumption above, the map $\phi$ in Theorem \ref{T:izom} is necessarily an isometry (distance preserving map) and what we need to prove is 'only' its surjectivity. One may think that this is not a big deal but, as we will see, it is highly nontrivial, we have to work quite hard to verify it.

Before presenting the proof of Theorem \ref{T:izom}, let us make its content more transparent by determining the structure of  the surjective isometries of $\B$. Let $\psi:\B \to \B$ be a surjective distance preserving map.
By the celebrated Mazur-Ulam theorem which tells that the surjective isometries between normed real-linear spaces are automatically affine, we have that the map $\omega(.)=\psi(.)-\psi(0)$ is a real-linear surjective isometry of $\B$. We claim that it is in fact either linear or conjugate-linear. To see this,
we make use of the result \cite[Corollary 3.3]{D} of Dang asserting that every surjective real-linear isometry of a $C^*$-algebra induces a decomposition of the algebra into the direct sum of two subalgebras such that the isometry in question is linear on the first subalgebra and conjugate-linear on the second. Clearly, $\B$ is not decomposable into the direct sum of two nontrivial subalgebras, hence we obtain that $\omega$ is either linear or conjugate-linear.
The structure of linear isometries of $\B$ is well-known. Namely, if $\omega$ is a linear surjective isometry of $\B$, then there are unitaries $U,V\in \B$ such that $\omega(A)=UAV$, $A\in \B$ or there are antiunitaries $U,V\in \B$ such that $\omega(A)=UA^*V$, $A\in \B$ (see, e.g., Theorem A.9 on page 208 in \cite{M}). If $\psi$ is a conjugate-linear surjective isometry of $\B$, then the map $A\mapsto \omega(A)^*$ is clearly a linear surjective isometry of $\B$ the structure of which is known.

Putting all these information together, we easily obtain that a map $\psi:\B\to \B$ is a surjective isometry if and only if there exist operators $U$ and $V$ on $H$ either both unitary or both antiunitary and an element $X\in \B$ such that $\psi$ is of one of the following two forms:
	\begin{equation*}
		\psi(A)=UAV+X, \quad A\in\B,
		\end{equation*}
		or
		\begin{equation*}
		\psi(A)=UA^*V+X,\quad A\in\B.
	\end{equation*}
	
One can now see that the content of Theorem \ref{T:izom} is exactly the following statement: If $\phi:\B\to \B$ is a map with the property that for any $A,B\in \B$ we have a pair $U_{A,B},V_{A,B}$ of either both unitary or both antiunitary operators on $H$ such that
\begin{equation}\label{E:ref1}
\phi(A)-\phi(B)=U_{A,B}(A-B)V_{A,B}
\end{equation}
or
\begin{equation}\label{E:ref2}
\phi(A)-\phi(B)=U_{A,B}(A-B)^*V_{A,B},
\end{equation}
then we necessarily have one single pair $U,V$ of either both unitary or both antiunitary operators on $H$ and an element $X\in \B$ such that
\[
\phi(A)=UAV +X, \quad A\in \B
\]
or
\[
\phi(A)=UA^*V +X, \quad A\in \B.
\]

After this discussion, we begin the proof of Theorem \ref{T:izom} with first presenting two auxiliary statements on which our proof rests. The first one is a result of Kuzma on the structure of additive maps decreasing rank one. Let $\F$ denote the algebra of all finite rank operators in $\B$. We say that an additive transformation $\psi:\F\to \F$ is decreasing rank one if $\psi$ maps rank-one operators to operators of rank at most one. We will also need the concept of quasilinearity of operators. Let $A:H\to H$ be an additive map and $h:\C\to\C$ be a nonzero ring homomorphism. We say that $A$ is $h$-quasilinear if $A(\lambda x)=h(\lambda) Ax$ holds for all $x\in H$ and $\lambda\in\C$. Let us introduce the following notation. For any $x,y\in H$, set
$$
L_x=\lbrace x\otimes y \vert y\in H \rbrace,\quad R_y=\lbrace x\otimes y \vert x\in H \rbrace.
$$
Here, for any $x,y\in H$, the symbol $x\otimes y$ stands for the rank at most one operator defined by $(x\otimes y)(z)=\langle z,y\rangle x$, $z\in H$.
Now, the theorem of Kuzma, namely \cite[Theorem 2.1]{K}, reads as follows.

\begin{theorem*}[Kuzma]
If $\psi:\F\to\F$ is an additive map which is decreasing rank one and its range is neither contained in any $L_x$ nor contained in any $R_y$, then $\psi$ is of one of the following two forms:
			\begin{equation*}\label{E:4}
				\psi(x\otimes y)=(Ax)\otimes(By), \quad x,y\in H,
			\end{equation*}
			or
			\begin{equation*}\label{E:5}
				\psi(x\otimes y)=(By)\otimes(Ax), \quad x,y\in H,
			\end{equation*}
			where $A,B:H\to H$ are $h$-quasilinear operators with some ring homomorphism $h:\C\to\C$.
	\end{theorem*}

The other ingredient of the proof of Theorem \ref{T:izom} is the following identification lemma.
In what follows let $\K$ denote the ideal of all compact linear operators on $H$.

\begin{lemma}\label{L:id}
If $T\in \B$ is an operator such that for every $K\in\K$ we have $\Vert I+K\Vert=\Vert T+K\Vert$, then $T=I$ necessarily holds.
\end{lemma}

\begin{proof}
Suppose that $T$ satisfies the assumption in the lemma. We clearly have $\Vert T\Vert=1$. Moreover, $2=\Vert I+P\Vert=\Vert T+P\Vert$ holds for any rank-one projection $P=x\otimes x$, $x\in H$ being an arbitrary unit vector. The equation $\Vert T+P\Vert=2$ implies that there exists a sequence $(y_n)$ of unit vectors in $H$ such that
		$$\Vert Ty_n+Py_n\Vert\to 2$$
as $n\to \infty$. Since $\Vert T\Vert=\Vert P\Vert=1$, we know that $\Vert Ty_n\Vert\le 1$ and $\Vert Py_n\Vert\le 1$ hold for all $n\in \N$. By the parallelogram identity, we infer
		$$\Vert Ty_n-Py_n\Vert^2+\Vert Ty_n+Py_n\Vert^2=2\Vert Ty_n\Vert^2+2\Vert Py_n\Vert^2.$$
Since the right hand side of this equation is less than or equal to $4$ and $\Vert Ty_n+Py_n\Vert^2\to 4$, we deduce that $\Vert Ty_n-Py_n\Vert\to 0$ and $\Vert Ty_n\Vert,\Vert Py_n\Vert\to 1$. It follows that
\begin{equation}\label{E:7}
Ty_n-Py_n=Ty_n-\langle y_n,x\rangle x\to 0
\end{equation}	
and, using $\Vert Ty_n\Vert \to 1$, we have $\vert\langle y_n,x\rangle\vert\to 1$. Because of the boundedness of the sequence $(y_n)$, it has a weakly convergent subsequence. Without loss of generality we may assume that already the original sequence $(y_n)$ is weakly convergent, $y_{n}\overset{w}{\to}z$ holds for some vector $z\in H$. Since $y_n$ is a unit vector for all $n$, we infer that $\| z\|\leq 1$ holds, too. From $\vert\langle y_n,x\rangle\vert\to 1$ we get that $|\langle z,x\rangle|=1$. Equality in Cauchy-Schwarz inequality implies linear dependence, hence we have $z=\varepsilon x$ for some complex number $\varepsilon$ of modulus 1. By \eqref{E:7}, we have $Ty_n \to \langle z,x\rangle x$. On the other hand, using $y_{n}\overset{w}{\to}z$, we also have $Ty_{n}\overset{w}{\to}Tz$. It follows that $Tz= \langle z,x\rangle x$ and, applying $z=\varepsilon x$, we conclude that $Tx=x$. Since $x$ was an arbitrary unit vector in $H$, we finally obtain that $T=I$.
\end{proof}

After these preliminaries we can now prove Theorem \ref{T:izom}.

\begin{proof}[Proof of Theorem \ref{T:izom}]
Let $\phi:\B \to \B$ be a map with the property that for any $A,B\in \B$ we have a pair $U_{A,B},V_{A,B}$ of either both unitary or both antiunitary operators on $H$ such that one of the following two equalities holds:
\begin{equation}\label{EE:1}
\begin{gathered}
\phi(A)-\phi(B)=U_{A,B}(A-B)V_{A,B}, \\
\phi(A)-\phi(B)=U_{A,B}(A-B)^*V_{A,B}.
\end{gathered}
\end{equation}
Clearly, without loss of generality we can assume that $\phi(0)=0$. By \eqref{EE:1}, it then follows that $\phi$ maps finite rank operators to finite rank operators and, in fact, $\phi$ preserves the rank. On the other hand, it also follows that $\phi$
preserves not only the operator norm distance but also the Hilbert-Schmidt norm distance on $\F$. The Hilbert-Schmidt norm originates from an inner product. In such spaces (even in any strictly convex space) isometries are automatically affine even without assuming their surjectivity, see \cite{Baker}. Since we have assumed $\phi(0)=0$, we have that $\phi$ is real-linear on $\F$ and it clearly maps rank-one operators to rank-one operators. We now apply Kuzma's theorem. Since our original map $\phi$ can be composed by the adjoint operation not affecting its local form \eqref{EE:1}, we may assume that there are a ring homomorphism $h:\C \to \C$ and $h$-quasilinear operators $A,B:H\to H$ such that $\phi(x\otimes y)=(Ax)\otimes (By)$ holds for all $x,y\in H$.
	
The real-linearity of $\phi$ on $\F$ easily implies that $h$ is the identity on the reals. It then follows that we have two possibilities: either $h(z)=z$ for all $z\in\C$ or $h(z)=\overline{z}$ for all $z\in\C$. We conclude that $A$ and $B$ are both linear or both conjugate-linear.
		
Our next aim is to show that $A$ and $B$ are either both unitary or both antiunitary. First observe that the injectivity of $\phi$ implies the injectivity of $A$ and $B$. Moreover, for any $x,y\in H$ we have
		$$\Vert x\Vert\Vert y\Vert=\Vert x \otimes y\Vert=\Vert\phi(x\otimes y)\Vert=\Vert Ax\otimes By\Vert=\Vert Ax\Vert\Vert By\Vert.$$
It follows that both $A,B$ are positive scalar multiplies of linear or conjugate-linear isometries and then we deduce that $A,B$ can be chosen to be both linear or both conjugate-linear isometries on $H$. We have
		\begin{equation*}
			\phi(x\otimes y)=A(x\otimes y)B^*,\quad x,y\in H.
		\end{equation*}
By the additivity of $\phi$ on $\F$ we obtain that
		\begin{equation*}
			\phi(F)=AFB^*, \quad F\in\F.
		\end{equation*}
Since $\phi$ is an isometry with respect to the distance coming from the operator norm and $\F$ is norm dense in $\K$, it follows that
		\begin{equation}\label{E:comp}
			\phi(K)=AKB^*,\quad K\in\K.
		\end{equation}
Select a complete orthonormal sequence $(e_n)$ in $H$ and consider the following compact operator
		$$K_0=\mathlarger{\mathlarger{\sum}\limits_{n\in\mathbb{N}}}\frac{1}{n}e_n\otimes e_n.$$
Clearly, $K_0$ is injective and has dense range. By the local form \eqref{EE:1} of $\phi$, the same is true for $\phi(K_0)$. On the other hand, by \eqref{E:comp}, we have
		\begin{equation*}
			\phi(K_0)=\mathlarger{\mathlarger{\sum\limits_{n\in\mathbb{N}}}}\frac{1}{n}Ae_n\otimes Be_n.
		\end{equation*}
We deduce that both sequences $(Ae_n)$, $(Be_n)$ generate dense subspaces in $H$ which means that $A,B$ are both unitaries or both antiunitaries.
	
After this, multiplying $\phi(.)$ by $A^*$ from the left and by $B$ from the right, we can clearly assume that $\phi$ is the identity on $\K$. In the last step of the proof we show that in that case $\phi$ equals the identity on the whole algebra $\B$, too.
To verify this, let $W$ be any unitary operator in $\B$, also let $\lambda\in\C$ and $K\in\K$ be arbitrary. Since $\phi$ is an isometry with respect to the metric of the operator norm, for every $K'\in\K$ the following equalities hold:
		$$\Vert\lambda W+(K-K')\Vert=\Vert\phi(\lambda W+K)-\phi(K')\Vert=\Vert\phi(\lambda W+K)-K'\Vert=\Vert\phi(\lambda W+K)-K+(K-K')\Vert.$$
In the case where $\lambda\neq 0$, this clearly implies
		$$\left\Vert I+\frac{1}{\lambda}W^*(K-K')\right\Vert=\left\Vert\frac{1}{\lambda}W^*(\phi(\lambda W+K)-K)+\frac{1}{\lambda}W^*(K-K')\right\Vert.$$
	As $K'$ runs through the whole set $\K$, the operator $(1/\lambda)W^*(K-K')$ also runs through it, so we can apply  Lemma \ref{L:id} and infer that
		$$I=\frac{1}{\lambda}W^*(\phi(\lambda W+K)-K).$$
This gives us
		$$\phi(\lambda W+K)=\lambda W+K,$$
(which trivially holds true also where $\lambda=0$) implying that $\phi$ acts as the identity on the operators of the form $\lambda W+K$, where $W,\lambda,K$ are as above. Using this, we can next prove that $\phi$ fixes the linear combinations of any two unitaries. Indeed, let $W,W'$ be unitary elements of $\B$ and $\lambda,\lambda'\in\C$ be arbitrary scalars. Then for every $K\in\K$ we have
		$$\Vert\lambda W+K\Vert=\Vert\phi(\lambda W+\lambda' W')-\phi(\lambda' W'-K)\Vert=\Vert\phi(\lambda W+\lambda' W')-\lambda' W'+K\Vert,$$
	and hence, assuming $\lambda\neq 0$, we infer
		$$\left\Vert I+\frac{1}{\lambda}W^*K\right\Vert=\left\Vert\frac{1}{\lambda}W^*(\phi(\lambda W+\lambda' W')-\lambda' W')+\frac{1}{\lambda}W^*K\right\Vert.$$
	Using the same reasoning as above, we deduce that
		$$I=\frac{1}{\lambda}W^*(\phi(\lambda W+\lambda' W')-\lambda' W'),$$
	which implies
		$$\phi(\lambda W+\lambda' W')=\lambda W+\lambda' W'$$
		and this holds true also when $\lambda=0$.
	One can continue with applying the above method and next derive that for any unitaries $W,W'\in\B$, complex numbers $\lambda,\lambda'\in\C$ and $K\in\K$ we have
		$$\phi(\lambda W+\lambda' W'+K)=\lambda W+\lambda' W'+K,$$
	and next that for any three unitaries $W_1,W_2,W_3\in \B$ and scalars $\lambda_1,\lambda_2,\lambda_3\in\C$ we have
		$$\phi(\lambda_1 W_1+\lambda_2 W_2+\lambda_3 W_3)=\lambda_1 W_1+\lambda_2 W_2+\lambda_3 W_3.$$
	In the last round we can prove that for any unitaries $W_1,W_2,W_3\in \B$, complex scalars $\lambda_1,\lambda_2,\lambda_3\in\C$ and $K\in\K$ we have
		$$\phi(\lambda_1 W_1+\lambda_2 W_2+\lambda_3 W_3+K)=\lambda_1 W_1+\lambda_2 W_2+\lambda_3 W_3+K,$$
	and finally that $\phi$ is fixing the linear combinations of any four unitaries in $\B$. But this exactly means that $\phi$ is the identity on the whole algebra $\B$ which finishes the proof of the theorem.
\end{proof}

After this, we can easily prove Theorem \ref{T:add}. Recall that any algebra *-automorphism of $\B$ is inner and implemented by a unitary element (see e.g., Theorem A.8 in \cite{M}).

\begin{proof}[Proof of Theorem \ref{T:add}]
Let $\phi:\B \to \B$ be a map which satisfies the requirements in the theorem.
It is apparent that for any $A\in \B$ there is a *-automorphism $\phi_{A,-A}$ of $\B$ such that $\phi(A)+\phi(-A)=\phi_{A,-A}(0)=0$. Hence we obtain that for any $A,B\in \B$, the equality
\[
\phi(A)-\phi(B)=\phi(A)+\phi(-B)=U_{A,B}(A-B)U_{A,B}^*
\]
holds with some unitary $U_{A,B}\in \B$. Therefore, the reformulation of Theorem \ref{T:izom} (see the statement above including the displayed formulas \eqref{E:ref1}, \eqref{E:ref2}), applies and, using also the easy fact that $\phi(0)=0$, we obtain that $\phi$ is of one of the forms
\[
\phi(A)=UAV, \quad A\in \B \quad \text{ or } \quad \phi(A)=UA^*V, \quad A\in \B,
\]
where either both $U,V$ are unitary or both of them are antiunitary operators on $H$. Since, from the original assumption on $\phi$ we see that $\phi(\lambda I)=\lambda I$ for every $\lambda \in \C$, it follows that we have either
\begin{equation*}\label{E:p1}
\phi(A)=UAU^*, \quad A\in \B
\end{equation*}
for a unitary operator $U$ on $H$, or we have
\begin{equation}\label{E:p2}
\phi(A)=UA^*U^*, \quad A\in \B
\end{equation}
for an antiunitary operator $U$ on $H$.

Assume now that $H$ is infinite dimensional. Since, by the original assumption on $\phi$, the operators
$\phi(A), A$ are unitarily similar for all $A\in \B$, it follows that for a unilateral shift $S$ on $H$, $\phi(S)$ is also a unilateral shift which immediately rules out the possibility \eqref{E:p2}. Consequently, it follows that $\phi$ is a *-automorphism.
Assume now that $H$ is finite dimensional. To treat this case we need to recall the following. For $n=2$, every $n$ by $n$ complex matrix is unitarily similar to its transpose but this is not true for any $n$ greater than 2. See 2.2.P3-2.2.P6 in \cite{HJ}. It then follows easily that if the dimension of $H$ is at least 3, the possibility \eqref{E:p2} is ruled out again, while in the 2-dimensional case it is not. The proof of Theorem \ref{T:add} is complete.
\end{proof}

We now turn to the proof of Theorem \ref{T:szor}. We will see that the argument is very different from the one in the proof of Theorem \ref{T:add}.

\begin{proof}[Proof of Theorem \ref{T:szor}]
Let $\phi:\B \to \B$ be a map with the property that
for any $A,B\in \B$ there is a unitary operator $U_{A,B}\in \B$ such that
\begin{equation}\label{E:20}
\phi(A)\phi(B)=U_{A,B}(AB)U_{A,B}^*.
\end{equation}
It is an immediate consequence of this property that $\phi(I)^2=I$, i.e., $S:=\phi(I)$ is an involution. Our first aim is to show that $S$ is self-adjoint.

To verify this, select an arbitrary rank-one (orthogonal) projection $P\in \B$. It follows from the property \eqref{E:20} that we have unit vectors $x,y,z\in H$ such that
		\begin{equation*}
			\phi(P)S=x\otimes x, \quad
			S\phi(P)=y\otimes y, \quad
			\phi(P)^2=z\otimes z.
		\end{equation*}
From the first two equations we deduce that
		\begin{equation*}
			\phi(P)=(x\otimes x)S=x\otimes (S^*x), \quad
			\phi(P)=S(y\otimes y)=(Sy)\otimes y.
		\end{equation*}
It follows that
		\begin{equation*}
\phi(P)^2=x\otimes (S^*x)\cdot(Sy)\otimes y=\langle Sy,S^*x\rangle x\otimes y.
		\end{equation*}
From the equality $z\otimes z=\phi(P)^2=\langle Sy,S^*x\rangle x\otimes y$
we get that the vectors $x,y,z$ are in the same 1-dimensional subspace. From
$x\otimes (S^*x)=\phi(P)=(Sy)\otimes y$ we see that $Sy$ is in the subspace generated by $x$ which equals the subspace generated by $y$. Therefore,
$Sy=\alpha y$ holds for some $\alpha\in\mathbb{C}$. Since
		$$\phi(P)^2=(Sy)\otimes y \cdot (Sy)\otimes y=\alpha^2(y\otimes y),$$
and $\phi(P)^2$ is a rank-one projection, we obtain that either $\alpha=1$ or $\alpha=-1$. This means that for any rank-one projection $P\in \B$, the operator $\phi(P)=\alpha (y\otimes y)$ is either a rank-one projection or its negative. We claim that this sign does not depend on the particular choice of $P$. Indeed, if we have two non-orthogonal rank-one projections $P_1$ and $P_2$ and two other rank-one projections $Q_1$ and $Q_2$ such that $\phi(P_1)=Q_1$ and $\phi(P_2)=-Q_2$, then applying \eqref{E:20} and using the trace functional $\tr$ we compute
		$$0<\tr P_1P_2=\tr \phi(P_1)\phi(P_2)=-\tr Q_1Q_2\leq 0,$$
which is a clear contradiction. If $P_1$ and $P_2$ are orthogonal, then we can choose a rank-one projection $P_3$ such that neither $P_1,P_3$ nor $P_3,P_2$ are  orthogonal and use the previous reasoning to verify our claim.
It follows that there is no serious loss of generality in assuming that for any rank-one projection $P\in \B$, the operator $\phi(P)$ is a rank-one projection (indeed, otherwise we consider the map $-\phi$). By \eqref{E:20}, we have
\begin{equation}\label{E:Wig}
\tr \phi(P)\phi(Q)=\tr PQ
\end{equation}
for any rank-one projections $P,Q$ on $H$. We next apply Wigner's famous theorem on quantum mechanical symmetry transformations which describes the structure of all self-maps of the set of all rank-one projections on $H$ with the property \eqref{E:Wig}, see, e.g., Theorem 2.1.4 in \cite{M}. It says that there is either a linear or a conjugate-linear isometry $J:H\to H$ such that
\begin{equation}\label{E:24}
\phi(P)=JPJ^*
\end{equation}
holds for every rank-one projection $P\in\mathcal{B}(H)$. Now, let $T\in\mathcal{B}(H)$ be an arbitrary self-adjoint operator. By the property \eqref{E:20}, there are self-adjoint operators $T_1,T_2\in\mathcal{B}(H)$ such that
		\begin{equation*}
			\phi(T)S=T_1, \quad
			S\phi(T)=T_2
		\end{equation*}
	which imply
		\begin{equation*}
			\phi(T)=T_1S, \quad
			\phi(T)=ST_2.
		\end{equation*}
We infer that
		$$\phi(T)^2=(T_1S)(ST_2)=T_1T_2.$$
On the other hand, by the property \eqref{E:20} again, $\phi(T)^2$ is clearly self-adjoint and hence we have that $T_1$ and $T_2$ commute. We then compute
		$$\phi(T)^2=(ST_2)(T_1S)=ST_1T_2S=S\phi(T)^2S$$
	and this gives us that
		$$S\phi(T)^2=\phi(T)^2S,$$
that is, $S$ and $\phi(T)^2$ also commute.
		Consider an orthonormal basis $(e_n)$ in $H$ and a strictly decreasing sequence $(\lambda_n)$ of positive real numbers converging to $0$. Define $T=\sum_n\lambda_ne_n\otimes e_n$. By \eqref{E:20}, $\phi(T)^2$ is of the following form:
		$$\phi(T)^2=\sum_{n=1}^\infty\lambda_n^2f_n\otimes f_n,$$
	where $(f_n)$ is also an orthonormal basis in $H$. As $S$ and $\phi(T)^2$ commute and the $\lambda_n^2$'s are all different, we easily obtain that $S$ commutes with each $f_n\otimes f_n$, $n\in \N$. As $\phi(I)=S$ is an involution, it has the form $S=I-2R$, where $R$ is an idempotent in $\mathcal{B}(H)$. It then follows that
		$$R\cdot f_n\otimes f_n=f_n\otimes f_n \cdot R.$$
We infer that for every $n\in\mathbb{N}$, there exists a scalar $\alpha_n\in\mathbb{C}$ such that
		$$Rf_n=\alpha_n f_n.$$
Applying $R$ on both sides, we get
		$Rf_n=\alpha_n^2 f_n$,
and it follows that $\alpha_n$ is either $0$ or $1$. Since $(f_n)$ is an orthonormal basis in $H$, we deduce that $R$ is an orthogonal projection and hence we obtain that $S$ is a self-adjoint involution. In particular, $\phi(I)=S$ is unitary.
	
	In the next step of the proof we will show that the image of any positive compact operator under $\phi$ is self-adjoint. Let $A\in\mathcal{B}(H)$ be of the form
		$$A=\sum_n \lambda_n e_n\otimes e_n,$$
where $(e_n)$ is an orthonormal basis in $H$ and $(\lambda_n)$ is a decreasing sequence of non-negative real numbers converging to zero. By the property \eqref{E:20}, there exist unitary operators $U,V,W\in\mathcal{B}(H)$ such that
		\begin{equation*}
			\phi(A)S=UAU^*, \quad
			S\phi(A)=VAV^*, \quad
			\phi(A)^2=WA^2W^*.
		\end{equation*}
	Let $B=WAW^*$. Then $A=W^* BW$, $\phi(A)^2=B^2$ and we have
		\begin{equation}\label{E:22}
			\phi(A)S=UW^*BWU^*, \quad
			S\phi(A)=VW^*BWV^*.
		\end{equation}
	Set $U'=UW^*$ and $V'=VW^*$. We obtain
		$$\phi(A)=U'BU'^*S=SV'BV'^*$$
	and
		$$B^2=\phi(A)^2=(U'BU'^*S)(SV'BV'^*)=U'BU'^*V'BV'^*.$$
Clearly, $B=\sum_n\lambda_n f_n\otimes f_n$ holds with an orthonormal basis $(f_n)$ in $H$. Consider the largest eigenvalue $\lambda_1$ of $B$ and the corresponding eigensubspace $M_1$. For any unit vector $x\in M_1$ we compute
$$
\begin{gathered}
		\| B\|^2=\Vert Bx\Vert^2=\langle B^2x,x\rangle=\langle U'BU'^*V'BV'^*x,x\rangle=\langle V'BV'^*x,U'BU'^*x\rangle\\\le\Vert V'BV'^*x\Vert\Vert U'BU'^*x\Vert=\|BV'^*x\|\|BU'^*x\|\leq \|B\|^2.
\end{gathered}	
$$
This gives us that
$$\|BV'^*x\|=\|B\|=\|BU'^*x\|.$$
Apparently, it follows that  $V'^*x,U'^*x\in M_1$.
Since $x\in M_1$ was an arbitrary unit vector in $M_1$, we have $V'^*(M_1),U'^*(M_1)\subset M_1$. In fact, because $M_1$ is finite dimensional, we actually obtain
$V'^*(M_1)=M_1=U'^*(M_1)$ and hence we also have
$V'(M_1)=M_1=U'(M_1)$. These imply that
		$$U'BU'^*\vert_{M_1}=V'BV'^*\vert_{M_1}=B\vert_{M_1}.$$	
Now, considering the orthogonal complement of $M_1$, restricting the operators $B, U'BU'^*, V'BV'^*$ to that subspace and repeating the previous argument, we obtain that $B, U'BU'^*, V'BV'^*$ coincide on the eigensubspace of $B$ corresponding to its second largest eigenvalue, and so forth. Therefore, we finally get that
		$$U'BU'^*=V'BV'^*=B.$$
	By \eqref{E:22}, this means that we have
\begin{equation}\label{E:23}
\phi(A)S=B,\quad  S\phi(A)=B.
\end{equation}	
	From this we deduce
		$S\phi(A)=\phi(A)S.$
	Since $S$ and $B$ are self-adjoint operators, using \eqref{E:23} we compute
		$$\phi(A)^*=(BS)^*=SB=\phi(A),$$
verifying that $\phi(A)$ is also self-adjoint.
	
	In the next step we show that on the set of positive Hilbert-Schmidt operators on $H$, $\phi$ is additive and positive homogeneous. It follows from \eqref{E:20} that $\phi$ sends Hilbert-Schmidt operators to Hilbert-Schmidt operators. Furthermore, if $A,B,\phi(A),\phi(B)\in \B$ are self-adjoint Hilbert-Schmidt operators, then we have $\langle \phi(A),\phi(B)\rangle_{HS}=\langle A,B\rangle_{HS}$, where $\langle .,.\rangle_{HS}$ denotes the Hilbert-Schmidt inner product. Now, for any positive Hilbert-Schmidt operators $A,B,C\in\mathcal{B}(H)$ and non-negative real number $\lambda$ we already know that $\phi(A+\lambda B),\phi(A),\phi(B)$ are self-adjoint and hence we can compute as follows
		\begin{equation*}
		\begin{gathered}
			\langle \phi(A+\lambda B)-(\phi(A)+\lambda\phi(B)),\phi(C)\rangle_{HS}=
		\langle \phi(A+\lambda B),\phi(C)\rangle_{HS}	-
			\langle \phi(A),\phi(C)\rangle_{HS}
			-\lambda\langle \phi(B)),\phi(C)\rangle_{HS}\\
			=\langle (A+\lambda B),C\rangle_{HS}	-
			\langle A,C\rangle_{HS}
			-\lambda \langle B,C\rangle_{HS}=0.
			\end{gathered}		
\end{equation*}
By the the real-linearity of the inner product in its second variable, it follows that
\[
\langle \phi(A+\lambda B)-(\phi(A)+\lambda\phi(B)),\phi(A+\lambda B)-(\phi(A)+\lambda\phi(B))\rangle_{HS}=0
\]
meaning that
\[
\phi(A+\lambda B)=\phi(A)+\lambda\phi(B)
\]
holds for any positive Hilbert-Schmidt operators $A,B\in\mathcal{B}(H)$ and non-negative real number $\lambda$. This gives us the additivity and positive homogeneity of $\phi$ on the set of all positive Hilbert-Schmidt operators on $H$.
	
	We already know that there exists a linear or conjugate-linear isometry $J:H\to H$ such that $\phi(P)=JPJ^*$ holds for every rank-one projection $P\in \B$, see \eqref{E:24}. Using what we have just proved above concerning the additivity and positive homogeneity of $\phi$, we can argue as follows. For an arbitrary orthonormal basis $(e_n)$ in $H$ and sequence $(\lambda_n)$ of non-negative real numbers which is square summable, we can compute
\begin{equation*}
\begin{gathered}
\phi\left(\sum_{n=1}^{\infty}\lambda_n e_n\otimes e_n\right)=\phi\left(\sum_{n=1}^{N}\lambda_n e_n\otimes e_n\right)+\phi\left(\sum_{n=N+1}^{\infty}\lambda_n e_n\otimes e_n\right),\\
\phi\left(\sum_{n=1}^{N}\lambda_n e_n\otimes e_n\right)=
\sum_{n=1}^{N}\lambda_n \phi(e_n\otimes e_n).
\end{gathered}
\end{equation*}
Applying \eqref{E:20} again and using the fact that $\phi(I)$ is unitary, it follows that
\begin{equation*}
\begin{gathered}
\left\Vert\phi\left(\sum_{n=1}^{\infty}\lambda_n e_n\otimes e_n\right)-\sum_{n=1}^{N}\lambda_n \phi(e_n\otimes e_n)\right\Vert=\left\Vert\phi\left(\sum_{n=N+1}^{\infty}\lambda_n e_n\otimes e_n\right)\right\Vert\\=\left\Vert\phi\left(\sum_{n=N+1}^{\infty}\lambda_n e_n\otimes e_n\right)\phi(I)\right\Vert=\left\Vert\sum_{n=N+1}^{\infty}\lambda_n e_n\otimes e_n\right\Vert\to 0
\end{gathered}
\end{equation*}
	as $N\to\infty$. Consequently, we have
\begin{equation*}
\begin{gathered}	
		\phi\left(\sum\limits_{n=1}^\infty\lambda_n e_n\otimes e_n \right)=\sum\limits_{n=0}^\infty\lambda_n\phi(e_n\otimes e_n)=\sum\limits_{n=0}^\infty\lambda_nJ (e_n\otimes e_n) J^*=J\left(\sum\limits_{n=0}^\infty\lambda_n e_n\otimes e_n\right)J^*.
		\end{gathered}
\end{equation*}
	By the property \eqref{E:20}, choosing nonzero $\lambda_n$'s we see that the operator $\phi(\sum_n \lambda_n e_n\otimes e_n )S$  has dense range. This implies that $\phi(\sum_n \lambda_n e_n\otimes e_n )$ has also dense range which ensures that our linear or conjugate-linear isometry $J$ has dense range, too. This implies that $J$ is either unitary or antiunitary.
	
To complete the proof, let $T\in \B$ be arbitrary. Pick any unit vector $x\in H$. Let $P=x\otimes x$. Using \eqref{E:20}, we have
\begin{equation*}
\tr J^*\phi(T)JP=\tr \phi(T)JPJ^*=\tr \phi(T)\phi(P)=\tr T P
\end{equation*}	
from which we obtain
\[
\langle J^*\phi(T)Jx, x\rangle =\langle Tx,x\rangle.
\]
Since $x$ was an arbitrary unit vector in $H$, it follows that $J^*\phi(T)J=T$ implying that $\phi(T)=JTJ^*$ holds for any $T\in \B$. In particular, it follows that $\phi(I)=I$. On the other hand, by \eqref{E:20}, we have
$\phi(iI)=\phi(iI)\phi(I)=iI$ and this implies that $J$ cannot be antiunitary, it is necessarily unitary.
This completes the proof of the theorem.
\end{proof}

We conclude the paper with the following. Firstly we remark that Peralta and his coauthors have recently considered another interesting generalization of the concept of 2-local maps that they called weak 2-locality, see, e.g., \cite{Cab, Li}. But still, their concept is, in some sense, closely related to the original one while ours here is very much different from that.

Let us look further and note that, as there have been serious investigations concerning 2-local automorphisms and 2-local isometries (2-local maps corresponding to the group of all linear isometries) of different algebras of operators and functions, it now seems to be a natural general problem to investigate questions  similar to the ones in the present paper in such algebras. The fact is that the first attempt has already been made, namely we refer to the recent preprint \cite{HatOi}. In that paper, motivated by the former general question (which was previously communicated to the authors), they have studied some function algebras and obtained results similar to our Theorem \ref{T:izom} for the algebra of all continuously
differentiable functions on the closed unit interval equipped with certain norms and also for the Banach
algebra of all Lipschitz functions on the closed unit interval with the sum-norm. At the end of their paper the authors have claimed that the analogous problem concerning the 'simplest' function algebra $C[0,1]$ seems to be really difficult. Sharing their claim, we think it kind of justifies our feeling that the general problem we have raised above may be an interesting direction of further research.

\bibliographystyle{amsplain}

\end{document}